\documentclass[11pt,letterpaper]{amsart}
\usepackage{enumitem}
\usepackage{color}
\usepackage[backref,hidelinks]{hyperref}
\newtheorem{thm}{Theorem}[section]

\newtheorem{lem}[thm]{Lemma}
\newtheorem{prop}[thm]{Proposition}
\numberwithin{equation}{section}

\theoremstyle{definition}

\newtheorem{rem}[thm]{Remark}

\newcommand{\C}{\mathbb{C}}
\newcommand{\Z}{\mathbb{Z}}

\newcommand{\A}{\mathbb{A}}

\newcommand{\K}{\mathbb{K}}
\newcommand{\m}{\mathfrak{m}}

\begin{document}

\title[]{Non-existence of negative weight derivations on Jacobian algebras}%
\author{Bingyi Chen}
\address{Bingyi Chen, Department of Mathematics, Sun Yat-sen University,\\
Guangzhou, 510275, P. R. China}
\email{chenby253@mail.sysu.edu.cn}

\author{Zhiwen Liu}
\address{Zhiwen Liu, Beijing Institute of Mathematical Sciences and Applications, Beijing, 101408, P. R. China.}
\email{liuzhiwen@bimsa.cn}

\author{Stephen S.-T.\ Yau} 
\address{Stephen S.-T.\ Yau, Department of Mathematical Sciences,
	Tsinghua University,
	Beijing, 100084, P. R. China}
\email{yau@uic.edu}

\author{Huaiqing Zuo}
\address{Huaiqing Zuo, Department of Mathematical Sciences,
	Tsinghua University,
	Beijing, 100084, P. R. China}
\email{hqzuo@mail.tsinghua.edu.cn}

\thanks{B. Chen was supported by NSFC Grant (No. 12501060) and the start-up fund
from Sun Yat-sen University. S. S.-T. Yau was supported by the Tsinghua University Education Foundation. H. Zuo was
supported by NSFC Grant (No. 12671056) and BJNSF Grant (No. 1252009).}


\begin{abstract}
Let $S=\K[x_1,\cdots,x_n]$ be a polynomial ring in $n$ variables with positive integer weights over a field $\K$ of characteristic zero. Let $f\in S$ be a weighted homogeneous polynomial such that the Jacobian algebra 
$$A=\left.S  \middle/  \left(\frac{\partial f}{\partial x_1},\cdots,\frac{\partial f}{\partial x_n}\right)\right.$$
is a finite-dimensional $\K$-vector space. We prove that $A$ has no nonzero negative weight derivation, which confirms a conjecture of the third author. An immediate consequence of this conjecture is the micro-local characterization of quasi-homogeneous singularities as shown in \cite{XY}.
\end{abstract}
\maketitle

\section{Introduction}
Let $(V,0)$ be a germ at the origin of a hypersurface in $\C^n$. Let $I(V)$ be the ideal of functions in $\C\{x_1,\cdots,x_n\}$ vanishing on $V$ and let $f$ be a generator of $I(V)$. It is well known that $0$ is an isolated singular point of $V$ if and only if  $$
 A(V):=\left.\C\{x_1,\cdots,x_n\}  \middle/  \left(f,\frac{\partial f}{\partial x_1},\cdots,\frac{\partial f}{\partial x_n}\right)\right.
$$
is a finite-dimensional $\C$-vector space.  Equipped with its natural $\C$-algebra structure, $A(V)$ is called the moduli algebra of $(V,0)$ (also known as the Tjurina algebra). In 1982, Mather and the third author \cite{MY} proved that two isolated hypersurface singularities of the same dimension are analytically isomorphic if and only if their moduli algebras are isomorphic. In 1983, the third author \cite{Yau1} introduced a finite-dimensional Lie algebra $L(V)$ to an isolated hypersurface singularity $(V,0)$, which is defined to be the Lie algebra of derivations of $A(V)$. In \cite{Yau2}, the third author  showed that $L(V)$ is actually a solvable Lie algebra. In \cite{SY}, Seeley and the third author used this Lie algebra structure to distinguish the analytic structures of a family of isolated hypersurface singularities. It is natural to ask how the algebra structures of $A(V)$ and $L(V)$ reflect the geometric structure of the singularity $(V,0)$.

Recall that an isolated hypersurface singularity $(V,0)\subseteq (\C^n,0)$ is called quasi-homogeneous if and only if 
$$f\in \left(\frac{\partial f}{\partial x_1},\cdots,\frac{\partial f}{\partial x_n}\right),$$
where $f$ is a generator of the ideal $I(V)$ of functions in $\C\{x_1,\cdots,x_n\}$ vanishing on $V$. By a theorem of Saito \cite{Sai}, if $(V,0)$ is quasi-homogeneous, after a biholomorphic change of coordinates, $f$ becomes a weighted homogeneous polynomial. In this case, the moduli algebra $A(V)$ is isomorphic to the Jacobian algebra
 $$
\left.\C[x_1,\cdots,x_n]  \middle/  \left(\frac{\partial f}{\partial x_1},\cdots,\frac{\partial f}{\partial x_n}\right)\right.
$$
and there is a grading on $A(V)=\bigoplus_{i\geq 0} A_i$. This grading induces a natural grading on the Lie algebra $L(V)=\bigoplus_{i\in \Z} L_i$ of derivations of $A(V)$, where $L_k$ consists of derivations that send $A_i$ to $A_{i+k}$. The third author conjectured that $L_k=0$ for all $k<0$, that is to say, there is no nonzero negative weight derivation on the moduli algebra of a quasi-homogeneous isolated hypersurface singularity. This conjecture is called Yau's conjecture. Assuming Yau's conjecture is true, Xu and the third author (\cite[Theorem B]{XY}) gave a micro-local characterization of quasi-homogeneity for isolated hypersurface singularities, only using the Lie algebra of derivations of the moduli algebra. Yau's conjecture has been confirmed in several special cases:
\begin{itemize}
  \item $n=2,3$ (\cite[Theorems 2.2 and 2.3]{CXY});
  \item $n=4$ (\cite[Theorem 2.1]{Che1});
  \item $w_n\geq w_1/2$, where $w_1$ (resp. $w_n$) is the largest (resp. smallest) weight of the variables (\cite[Main Theorem]{YZ}); 
  \item $\deg f>w_1^2-w_1+1$, where $w_1$ is the largest weight  of the variables  (\cite[Theorem 1.5]{CY}).
\end{itemize}

The main result of this paper is a proof of Yau's conjecture. An immediate consequence of this result is the micro-local characterization of quasi-homogeneous singularities as shown in \cite[Theorem B]{XY}.

\begin{thm} \label{mthm}
Let $S=\K[x_1,\cdots,x_n]$ be a polynomial ring in $n$ variables with positive integer weights over a field $\K$ of characteristic zero. Let $f\in S$ be a weighted homogeneous polynomial such that
$$A=\left.S  \middle/  \left(\frac{\partial f}{\partial x_1},\cdots,\frac{\partial f}{\partial x_n}\right)\right.$$
is a finite-dimensional $\K$-vector space. Then there is no nonzero negative weight derivation on $A$.
\end{thm}
\begin{rem}
If $\K$ is of characteristic $p$ with $p>0$, Theorem \ref{mthm} does not hold. For example, let $S=\K[x]$, $f=x^{p+1}$ and $D=\partial/\partial x$. Then $D$ induces a nonzero negative weight derivation on the Jacobian algebra.
\end{rem}

Yau's conjecture is a special case of Halperin's conjecture, which is one of the most important open problems in rational homotopy theory. Let $A=\K[x_1,\cdots,x_n]/(f_1,\cdots,f_n)$ be a positively graded complete intersection Artinian algebra, where $\K$ is a field of characteristic zero. The grading on $A$ induces a natural grading on the Lie algebra of derivations of $A$. Halperin (\cite[p.516]{FHT}) conjectured that $A$ has no nonzero negative weight derivation. A positive answer to this conjecture implies the collapse of the rational Serre spectral sequence at the $E_2$ page for any orientable fibration whose fiber is an elliptic space and has no cohomology in odd degrees. Some partial results for Halperin's conjecture can be found in \cite{Tho}, \cite{ST}, \cite{Lup}, \cite{Mar}, \cite{AM}, \cite{PP}, \cite{Che2}, \cite{CYZ}, \cite{KW}. Theorem \ref{mthm} is also a partial result for Halperin's conjecture in the case when $f_i=\partial f/\partial x_i$ for a single polynomial $f$.

In a recent paper \cite{Ale}, Halperin's conjecture was claimed to be proved. However, the proof given there contains a gap. \footnote{The argument in \cite[\S8, pp. 27-28]{Ale} uses a
contraction-extension step that does not preserve zero-dimensionality, leaving its application of Corollary 7.6 unjustified. For example, let
$R=\C\{x\}$, $S=\C\{x,y\}$, and $I=(x^2,y^2)\subset S$.
Then $S/I$ is a zero-dimensional complete intersection algebra, whereas
$I^c=I\cap R=(x^2)$ and $I^{ce}=I^cS=(x^2)S$, so
$\dim(S/I^{ce})=1$.
Thus the extended ideal does not satisfy the
zero-dimensionality hypothesis required by Corollary 7.6.}
Our proof is quite different from that in \cite{Ale} and has been obtained independently.

\subsection*{AI disclosure} 
The initial proof of Theorem \ref{mthm} was obtained with the assistance of ChatGPT. The final proof in this paper has been generated and reviewed by the authors. ChatGPT was also used to assist with checking mathematical arguments and improving the exposition and language. The authors take full responsibility for the mathematical content, references, and final presentation of the paper.

\section{Preliminaries}
In this section, we recall some basic definitions, notation and facts about derivations, graded algebras and polynomials. We always assume that $\K$ is a field of characteristic zero. Let $S=\K[x_1,\cdots,x_n]$ be a polynomial ring in $n$ weighted variables $x_1,\cdots,x_n$ with positive integer weights $w_1,\cdots,w_n$. Denote by $\m$ the maximal ideal of $S$ generated by $x_1,\cdots,x_n$.

\subsection{Graded algebras and derivations}
Let $A$ be a commutative algebra over $\K$.  A derivation $D$ on $A$ is a $\K$-linear endomorphism of $A$ satisfying the Leibniz rule:
\begin{equation}
D(ab)=D(a)b+aD(b) \quad \text{for any } a,b\in A.
\end{equation}

We say a commutative algebra $A$ is a graded algebra if $A=\oplus_{i\geq 0} A_i$, where $A_i$ are linear subspaces of $A$ such that $A_i\cdot A_j\subseteq A_{i+j}$ for any $i,j\geq 0$. We set $A_j=0$ for $j<0$. If a derivation $D$ on a graded algebra $A=\oplus_{i\geq 0} A_i$ satisfies $D(A_i)\subseteq A_{i+k}$ for any $i\geq 0$, we say $D$ is of weight $k$.

\subsection{Weighted polynomial rings}
For any monomial $\prod_{i=1}^n x_i^{a_i}\in S$, its (weighted) degree is defined to be $\sum_{i=1}^n a_iw_i$. Denote by $S_i$ the linear subspace of $S$ spanned by monomials of degree $i$. 
Then $S=\oplus_{i\geq 0} S_i$ is a graded algebra. 

Let $f\in S$ be a polynomial. Denote by $J_f$ the ideal of $S$ generated by  $\partial f/\partial x_i$, $i=1,\cdots,n$, which is called the Jacobian ideal of $f$. If $f\in S_k$ for some $k$, we say $f$ is a weighted homogeneous polynomial of (weighted) degree $k$ and in the case when $f\neq 0$ we denote $k$ by $\deg f$. 

Every derivation $D$ on $S$ can be written as 
\begin{equation*}
D=p_1\partial/\partial {x_1}+\cdots+p_n \partial/\partial {x_n},
\end{equation*}
where $p_i\in S$.  Then $D$ is of weight $k$ if and only if $p_i$ is either the zero polynomial or a weighted homogeneous polynomial with $\deg p_i-w_i=k$ for $i=1,\cdots,n$.

\subsection{Lifting derivations}
Let $J$ be an ideal of $S$ generated by weighted homogeneous polynomials $f_1,\cdots,f_m$. Then $S/J=\oplus_{i\geq 0} (S/J)_i$ is a graded algebra, where $(S/J)_i$ is the image of $S_{i}$ under the projection $S\rightarrow S/J$.  The following well-known lemma tells us that every derivation on $S/J$ of weight $k$ can be lifted to a derivation on $S$ of weight $k$.

\begin{lem}[{\cite[Lemma 2.1]{CY}}]\label{lem:lift}
Let $\overline{D}$ be a derivation on $S/J$ of weight $k$. Then there exists a derivation $D$ on $S$ of weight $k$ such that

\begin{enumerate}
\item[\textup{(1)}] $\overline{D(g)}=\overline{D}\big(\overline{g}\big)$ for any $g\in S$;
\item[\textup{(2)}] $D(J)\subseteq J$.
\end{enumerate}
Here we denote by $\overline{g}$ the image of $g$ under the projection $S\rightarrow S/J$ for any $g\in S$.
\end{lem}

\subsection{\texorpdfstring{$\m$}{m}-primary ideals}

\begin{lem}\label{lem:primary}
Let $J$ be an ideal of $S$ generated by some weighted homogeneous polynomials $f_1,\cdots,f_m$. Assume that $S/J$ is a finite-dimensional $\K$-vector space. Then there is an integer $N>0$ such that $\m^N\subseteq J$. 
\end{lem}

\begin{proof}
If $J=S$, the conclusion is immediate. So we may assume that $J\neq S$. Let $N$ be the dimension of $S/J$ as a $\K$-vector space. Then $N>0$ as $J\neq S$. Assume that $x_i^N\notin J$ for some $i$. Denote by $\overline{x_i}$ the image of $x_i$ under the projection $S\to S/J$. If 
$$\sum_{j=0}^N c_j \overline{x_i}^{j}=0$$
in $S/J$ for some $c_j\in \K$, then $c_j \overline{x_i}^{j}=0$ and hence $c_j=0$ for every $j$, since the elements $\overline{x_i}^{j}$, $0\leq j\leq N$, are nonzero and have distinct degrees
(as $S/J$ is a graded algebra, an element in $S/J$ is zero if and only if all its homogeneous components are zero).
So $\overline{x_i}^{0},\overline{x_i}^{1},\cdots,\overline{x_i}^{N}$ are linearly independent in $S/J$, a contradiction. Therefore, $x_i^N\in J$ for each $i$, which implies that $\m^{n(N-1)+1}\subseteq J$ by the pigeonhole principle.
\end{proof}

\subsection{Koszul complex}

\begin{lem}\label{lem:koszul}
Let $f_1,\cdots,f_n\in \mathfrak m \subseteq S$ be $n$ weighted homogeneous polynomials such that $S/(f_1,\cdots,f_n)$ is a finite-dimensional $\K$-vector space. If 
\begin{equation}\label{Keq1}
\sum_{i=1}^n p_i f_i=0
\end{equation}
where $p_1,\cdots,p_n\in S$,
then there are  $c_{ij}=-c_{ji}\in S$ such that
\begin{equation}\label{Keq2}
 p_i=\sum_{j=1}^n c_{ij}f_j\quad\text{for } i=1,\cdots,n.
\end{equation}
In particular, $p_i\in (f_1,\cdots,f_n)$ for every $i$.
\end{lem}

\begin{proof}
By Lemma \ref{lem:primary}, there is an integer $N>0$ such that
\[
\m^N\subseteq (f_1,\cdots,f_n)\subseteq \m.
\]
Thus $f_1,\cdots,f_n$ form a system of parameters of the regular local ring $S_{\mathfrak m}$. Since \(S_{\mathfrak m}\) is Cohen-Macaulay, they form a regular sequence in $S_{\mathfrak m}$.

We claim that $f_1,\ldots,f_n$ form a regular sequence in $S$.
For each integer $r$ with $0\leq r\leq n-1$, let
$J=(f_1,\ldots,f_r)$, with the convention that $J=(0)$ when $r=0$.
It suffices to show that the image $\overline{f_{r+1}}$ of $f_{r+1}$ in $S/J$
is not a zero divisor. Suppose that $f_{r+1}g\in J$ for some $g\in S$. As $f_1,\cdots,f_n$ form a regular sequence in $S_{\m}$, $f_{r+1}$ is not a zero divisor in $S_{\m}/J_{\m}$, so $g\in J_{\m}$, i.e. there are $a_i\in S,b\in S\setminus {\m}$ such that
$$bg=\sum_{i=1}^{r} a_if_i.$$
So $\overline{b}\overline{g}=0$ in $S/J$, where $\overline{b},\overline{g}$ are the images of $b,g$ under the projection $S\to S/J$. Note that $S/J$ is a graded algebra.  If $\overline{g}\neq 0$, let $\overline{g}_d$ be its nonzero homogeneous component of least degree. As $b\in S\setminus {\m}$, $b$ has a nonzero constant term $k$. So the homogeneous component of $\overline{b}\overline{g}$ of least degree is $k\overline{g}_d\neq 0$, which implies that $\overline{b}\overline{g}\neq 0$, a contradiction. Thus $\overline{g}=0$ and $\overline{f_{r+1}}$ is not a zero divisor in $S/J$. Therefore, $f_1,\cdots,f_n$ form a regular sequence in $S$.

So the Koszul complex on $f_1,\cdots,f_n$ is exact in positive homological degrees. Equation \eqref{Keq1} gives that $(p_1,\cdots,p_n)\in S^n$ is a cycle in homological degree one, so it is a Koszul boundary, giving the skew-symmetric formula \eqref{Keq2}; see \cite[Section 17.4]{Eis}.
\end{proof}

\subsection{Flows of negative weight derivations}\label{sec:flow}
Let $D$ be a  derivation on $S$ of weight $-\rho<0$. 
Then $D$ is locally nilpotent. More precisely, if $g\in S$ is a weighted homogeneous polynomial of degree $m$, then
$$
 D^N(g)=0\quad\text{ if }N\rho>m.
$$
For any $t\in \K$, the exponential
$$
\exp(tD)=\sum_{\ell\geq 0}\frac{t^\ell D^\ell}{\ell!}
$$
defines a polynomial automorphism of $S$, with inverse $\exp(-tD)$ (as $D$ is locally nilpotent, the sum is finite on each polynomial, so it is well-defined). The induced automorphism $\Phi_{D,t}$ of $\A^n_{\K}$ is given by 
$$(\Phi_{D,t})_i(a)=\exp(tD)(x_i)(a)$$
for any $a\in \A^n_{\K}$ and $i=1,\cdots,n$, which is called the flow of $D$.

\section{Proof of the main result}
Throughout this section, we retain the notation in Section 2. 

First we reduce Theorem \ref{mthm} to the case where $\deg f\geq 2w_i$ for every $i=1,\cdots,n$. Analogous techniques were used in \cite[proof of Satz 1.3]{Sai}.

\begin{lem}\label{lem:splitting}
Let $f\in S$ be a weighted homogeneous polynomial of degree $d$ such that $J_f\subseteq \m$ and $S/J_f$ is a finite-dimensional $\K$-vector space. Then after a graded polynomial change of coordinates,
\begin{equation}\label{eq:splitting}
    f = u_1v_1 + \cdots + u_rv_r + h(z_1,\ldots,z_m),
\end{equation}
where
$$
\deg u_i > \frac{d}{2}, \quad
\deg v_i = d - \deg u_i < \frac{d}{2}, \quad
\deg z_j \leq \frac{d}{2}.
$$
Moreover, there is a graded
algebra isomorphism
$$
   S/J_f \cong  \K[z_1,\ldots,z_m]/J_h.
$$
If $m=0$, then $S/J_f \simeq\K$.
\end{lem}

\begin{proof}
Write $f_i=\partial f/\partial x_i$ for any $i=1,\cdots,n$. By Lemma \ref{lem:primary}, there is an integer $N>0$ such that $\m^N\subseteq (f_1,\cdots,f_n) \subseteq \m$. 

Call a variable high (resp. low) if its weight is greater than $d/2$ (resp. at most $d/2$).
Write the high variables as $u_1,\ldots,u_r$ and all other variables
as $y$. A high variable cannot have weight at least $d$: if its weight is greater than $d$, it is absent from $f$, and if its weight equals $d$, its partial derivative is either zero or a nonzero constant. Neither case is possible as  $\m^N\subseteq (f_1,\cdots,f_n) \subseteq \m$.

No monomial of degree $d$ can contain two high
variables, or a square of a high variable. Hence
\begin{equation*}
f = \sum_{i=1}^{r} u_i g_i(y) + h_0(y),
\end{equation*}
where each $g_i$ is weighted homogeneous of degree $d - \deg u_i\in (0,d/2)$.
For each $i$, let $\ell_i$ be the ordinary linear part of $g_i$; it is a linear combination of variables of weight $d-\deg u_i$. We claim that the $\ell_1,\cdots,\ell_r$ are linearly independent.
Otherwise, choose constants $c_i$, not all zero, with
$$ \sum_{i=1}^r c_i \ell_i = 0. $$
Consider the line on which $u_i = tc_i$ for $i=1,\cdots,r$ and every other coordinate is zero. The derivatives with respect to the high
variables vanish there because $g_i(0) = 0$. The derivative with
respect to a low variable $y_j$ on the line is
\[
    t \sum_{i=1}^r
        c_i \frac{\partial g_i}{\partial y_j}(0)
    + \frac{\partial h_0}{\partial y_j}(0)
    = 0.
\]
The first term vanishes by the chosen dependence, and the second
vanishes because $h_0$ is weighted homogeneous of degree $d$ while every low
variable has weight at most $d/2$, so $h_0$ has no ordinary linear
term. Thus all partial derivatives vanish on the whole line, contradicting the fact that $\m^N\subseteq (f_1,\cdots,f_n)$.

Within the span of the low variables of each fixed weight, extend the corresponding $\ell_i$ to a basis. This gives a graded linear change of coordinates in which the $\ell_i$ become distinct coordinate variables $v_i$. Denote the remaining low variables by $z$. Then
$$ g_i = v_i + r_i(v,z),$$
where $r_i$ has no ordinary linear term and is weighted homogeneous of degree $\deg v_i$. So every variable occurring in $r_i(v,z)$ has weight strictly less than $\deg v_i$. Thus the change of coordinates $v'_i=v_i+r_i(v,z)$ is triangular with respect to increasing weights; its inverse is polynomial by successive substitution. It is therefore a graded polynomial automorphism. Relabeling the new coordinates as $v_i$, we obtain
$$ f = \sum_{i=1}^{r} u_i v_i + h_1(v,z).$$
Put $h(z) = h_1(0,z)$. Since $h_1(v,z)-h(z)$ lies in
$(v_1,\ldots,v_r)$, write
$$ h_1(v,z) = h(z) + \sum_{i=1}^{r} v_i a_i(v,z), $$
where each $a_i$ is weighted homogeneous of degree $\deg u_i$.
The graded coordinate change $u'_i=u_i+a_i(v,z)$ gives \eqref{eq:splitting} after relabeling the new coordinates as $u_i$.

Now
$$J_f = (v_1,u_1,\ldots,v_r,u_r,J_h),$$
and hence
$$S/J_f \cong \K[z_1,\ldots,z_m]/J_h$$
as graded algebras.
\end{proof}

The following proposition is crucial in the proof of Theorem \ref{mthm}.

\begin{prop}\label{prop:degree}
Let $f\in S$ be a weighted homogeneous polynomial of degree  $d$ such that $d\geq 2w_i$ for any $i=1,\cdots,n$ and $S/J_f$ is a finite-dimensional $\K$-vector space. Let $g\in S$ be a nonzero weighted homogeneous polynomial of degree $e$ such that $e>0$ and $J_g \subseteq J_f$. Then $e\geq d$.
\end{prop}

\begin{rem}
Proposition \ref{prop:degree} appears as Assertion 1 in \cite{Sho}. However, the proof given there contains a gap. \footnote{
The prefactor in \cite[p.~132, equation (5)]{Sho} should be
$$
\frac{t^{n-(\alpha_1+\cdots+\alpha_n)}}
     {t^{n\lambda(\varphi)-(\alpha_1+\cdots+\alpha_n)}},
$$
rather than
$$
\frac{t^{n-(\alpha_1+\cdots+\alpha_n)}}
     {t^{m+k\lambda(\varphi)-(\alpha_1+\cdots+\alpha_n)}}.
$$
Indeed, under the weighted change of variables used there,
each denominator factor contributes
$t^{\lambda(\varphi)-\alpha_i}$, with the remaining factor
being $\varphi_i+t^pF_i$, where $p=1-\lambda(\varphi)$.
This also holds when $\varphi_i=0$.
Thus the corrected prefactor is $t^{np}$ rather than $t^{kp}$.
With this correction, the comparison of powers of $t$
no longer implies $k=n$.
Consequently, the argument given there does not establish
Assertion 1.
} We give a new proof here, which is quite different from that in \cite{Sho}.
\end{rem}
\begin{proof}
Suppose, for a contradiction, that $e<d$. Write $f_i=\partial f/\partial x_i$ and $g_i=\partial g/\partial x_i$ for any $i=1,\cdots,n$. As $\deg g=e>0$, some partial derivative of $g$ is nonzero. Choose an index $i$ such that $g_i\neq 0$ and $w_i\geq w_j$ for any index $j$ with $g_j\neq 0$. Set
$$u=w_i\quad \text{ and } \quad  r=d-e>0.$$
As $g_i\in J_g\subseteq J_f$, we can write
\begin{equation}\label{eq1}
 g_i=\sum_{j=1}^n a_j f_j,
\end{equation}
where each $a_j$ is either zero or weighted homogeneous with
\begin{equation}\label{eq2}
\deg a_j=(e-w_i)-(d-w_j)=w_j-u-r.
\end{equation}
Define a derivation
$$
 V=\sum_{j=1}^n a_j\partial_{x_j}.
$$
By \eqref{eq1}, we have
\begin{equation}\label{eqq}
 V(f)=g_i\neq 0.
\end{equation}
Thus $V$ is a nonzero derivation. By \eqref{eq2}, $V$ is of weight $-(u+r)<0$.

If $a_j\neq 0$, its degree is nonnegative and hence \eqref{eq2} gives 
\begin{equation}\label{eq3}
w_j\geq u+r>u.
\end{equation}
Therefore, for any index $k$ with $w_k\leq u$, $a_k=0$ and hence $V(x_k)=0$. As $V\neq 0$, at least one $a_j$ is nonzero, hence \eqref{eq3} gives
$$u<w_j\leq \frac{d}{2}.$$

By the choice of $u$, $g_j=0$ for any index $j$ with $w_j> u$.
In characteristic zero, this means that $g$ is independent of every
variable of weight greater than $u$. Its derivative $g_i$ is therefore also independent of these variables, so
$$V(f)=g_i\in\K[x_k:w_k\le u].$$
The above observations show that all hypotheses of Lemma \ref{lem:flow} below hold. That lemma gives $V=0$, contradicting \eqref{eqq}.
Therefore $e\geq d$.
\end{proof}

\begin{lem}\label{lem:flow}
Let $f\in S$ be a weighted homogeneous polynomial of degree $d$ such that $d\geq 2w_i$ for any $i=1,\cdots,n$ and $S/J_f$ is a finite-dimensional $\K$-vector space. Let $V$ be a derivation on $S$ of weight $-\rho<0$. Assume that, for some $u< d/2$,
\begin{enumerate}
\item[\textup{(1)}] $V(x_i)=0$ whenever $w_i\leq u$;
\item[\textup{(2)}] $V(f)\in \K[x_i:w_i\leq u]$.
\end{enumerate}
Then $V=0$.
\end{lem}

\begin{proof}
By Lemma \ref{lem:primary}, there is a positive integer $N$ such that $\m^N\subseteq J_f$. Suppose, for a contradiction, that $V\neq 0$. We will construct a polynomial curve $P:\A^1_{\K}\to \A^n_{\K}$ by combining weighted scaling and the flow of $V$ such that
$P(0)$ is a common zero of all partial derivatives of $f$
and is distinct from the origin. This contradicts
$\mathfrak m^N\subseteq J_f$.

Set $h=V(f)$. Write $f_i=\partial f/\partial x_i$ and $h_i=\partial h/\partial x_i$ for $i=1,\cdots,n$. Set 
$$
 N_i=\max\{\ell\geq 0\mid V^\ell(x_i)\neq 0\}
$$
for $i=1,\cdots,n$. Thus $N_i=0$ exactly when $V(x_i)=0$. There is at least one $N_i\neq 0$ because $V\neq 0$. 
Set
\begin{equation*}
 \lambda=\min_{N_i>0}\frac{w_i}{N_i}\quad \text{ and }\quad 
 \alpha=\lambda-\rho.
\end{equation*}
Since $V^{N_i}(x_i)$ has degree $w_i-N_i\rho\geq 0$, one has $N_i\rho\le w_i$ for any $i$, hence $\alpha=\lambda-\rho\geq 0$.
Choose a positive integer $c$ such that both $c\alpha$ and $c\lambda$ are integers. Choose an index $i_0$ such that $\lambda=w_{i_0}/N_{i_0}$.  As $\K$ is infinite (because $\K$ has characteristic zero) and as the polynomial $V^{N_{i_0}}(x_{i_0})$ is nonzero, we can choose
$a=(a_1,\ldots,a_n)\in\A_{\K}^n$ outside its zero set. For any $0\neq \tau\in \K$, define
\begin{align*}
Q(\tau)&=(\tau^{cw_1}a_1,\ldots,\tau^{cw_n}a_n),\\
t(\tau)&=\tau^{-c\alpha} \quad \text{ and  }\quad 
P(\tau)=\Phi_{t(\tau)}(Q(\tau)),
\end{align*}
where $\Phi_{t}=\Phi_{V, t}$ is the flow of the negative weight derivation $V$ (see Subsection \ref{sec:flow}).

As $V^\ell(x_i)$ is weighted homogeneous of degree $w_i-\ell\rho$, we have
$$
 V^\ell(x_i)(Q(\tau))
 =\tau^{c(w_i-\ell\rho)}V^\ell(x_i)(a).
$$
It follows that 
\begin{equation}\label{neqP}
 P_i(\tau)=\sum_{\ell=0}^{N_i}
 \frac{\tau^{c(w_i-\ell\lambda)}}{\ell!}V^\ell(x_i)(a),
\end{equation}
where $P_i(\tau)$ is the $i$-th coordinate of $P(\tau)$. For any $0\leq \ell\leq N_i$,
$$
 w_i-\ell\lambda\ge w_i-N_i\lambda\geq 0.
$$
Thus every exponent in \eqref{neqP} is a nonnegative integer
and hence $P_i(\tau)\in\K[\tau]$.
In particular, $P$ extends to a polynomial map
$\A^1_{\K}\to \A^n_{\K}$. Set
$$P_0=P(0).$$
For $i=i_0$, the exponent in \eqref{neqP} is zero only when
$\ell=N_{i_0}$. So we have 
\begin{equation}\label{eqn0}
 (P_0)_{i_0}=\frac{V^{N_{i_0}}(x_{i_0})(a)}{N_{i_0}!}\ne0
\end{equation}
by the choice of $a$. Hence $P_0$ is distinct from the origin.

Conditions (1) and (2) imply that $V(h)=0$. It follows that
\begin{equation*}
f\circ\Phi_t=\exp(tV)(f)=f+th,
\end{equation*}
and hence
\begin{equation*}
f=(f+th)\circ \Phi_{-t},
\end{equation*}
as $\Phi_{-t}$ is the inverse of $\Phi_t$. Write $\nabla f=(f_1,\ldots,f_n)$ as a row vector and let
\[
 K(\tau)=J\Phi_{-t(\tau)}(P(\tau))
\]
be the Jacobian matrix of $\Phi_{-t(\tau)}:\A^n_{\K}\to \A^n_{\K}$ at $P(\tau)$.
Differentiating $f=(f+th)\circ \Phi_{-t}$ with respect to $x$ at fixed
$t$, and then setting $x=P(\tau)$ and $t=t(\tau)$, gives
\begin{equation*}
 \nabla f(P(\tau))=v(\tau) K(\tau)
\end{equation*}
where
$$
 v(\tau)=\nabla f(Q(\tau))+t(\tau)\nabla h(Q(\tau)).
$$
Therefore
\begin{equation}\label{neqhahah}
 f_k(P(\tau))=\sum_{i=1}^n v_i(\tau)K_{ik}(\tau)
\end{equation}
for $k=1,\cdots,n$.

The $i$-th coordinate of $\Phi_{-t}$ is
$$
 \exp(-tV)(x_i)=\sum_{\ell=0}^{N_i}\frac{(-t)^\ell}{\ell!}V^\ell(x_i).
$$
It follows that
\begin{equation*}
 K_{ik}(\tau)=\sum_{\ell=0}^{N_i}
 \frac{(-t(\tau))^\ell}{\ell!}
 \partial_{x_k}\bigl(V^\ell(x_i)\bigr)(P(\tau)).
\end{equation*}
As $P_j(\tau)\in \K[\tau]$ for any $j$ and $t(\tau)=\tau^{-c\alpha}$, we have
\begin{equation}\label{neqK}
K_{ik}(\tau)\in \tau^{-cN_i\alpha}\cdot \K[\tau].
\end{equation}

Suppose first that $w_i>u$.  Since $h$ involves only variables of weight at most $u$ by condition (2), one has $h_i=0$.
Hence
$$
 v_i(\tau)=f_i(Q(\tau))=\tau^{c(d-w_i)}f_i(a).
$$
By \eqref{neqK}, for every $k$,
$$
 v_i(\tau)K_{ik}(\tau)\in \tau^{c(d-w_i-N_i\alpha)}\cdot \K[\tau].
$$
We claim that the exponent $c(d-w_i-N_i\alpha)$ is positive and hence $v_i(\tau)K_{ik}(\tau)\in \tau \cdot\K[\tau]$. Indeed, if $N_i=0$, we have 
$$d-w_i-N_i\alpha=d-w_i\geq w_i>0;$$
if $N_i>0$, we have
\begin{align*}
 d-w_i-N_i\alpha
 &=d-w_i-N_i\lambda+N_i\rho\\
 &\ge d-2w_i+N_i\rho>0.
 \end{align*}
Here we used $N_i\lambda\le w_i$, $w_i\le d/2$, $w_i>0$ and $\rho>0$.

Now suppose that $w_i\le u$. Condition (1) gives $N_i=0$ and hence 
$K_{ik}(\tau)\in \K[\tau]$ by \eqref{neqK}. Note that
 \begin{align*}
 f_i(Q(\tau))&=\tau^{c(d-w_i)}f_i(a), \text{ and }\\
 t(\tau)h_i(Q(\tau))
 &=\tau^{c(d-\rho-w_i-\alpha)}h_i(a)\\
 &=\tau^{c(d-w_i-\lambda)}h_i(a).
 \end{align*}
As $V\neq 0$, there is at least one $x_j$ with $V(x_j)\neq 0$
and hence $N_j\geq 1$. So
$$\lambda\le\frac{w_j}{N_j}\le w_j\le\frac{d}{2}.$$
Since $w_i\le u<d/2$, we have
\begin{equation}\label{eq:fixed-exponent}
 d-w_i-\lambda>0.
\end{equation}
Consequently $v_i(\tau)K_{ik}(\tau)$ also belongs to $\tau \cdot\K[\tau]$ in this case.

Therefore, $v_i(\tau)K_{ik}(\tau)\in \tau \cdot\K[\tau]$ for any $i,k$. Equation \eqref{neqhahah} gives
$$
 \nabla f(P_0)=0.
$$
This contradicts $P_0\neq 0$ and $\mathfrak m^N\subseteq J_f$ for some $N>0$.
\end{proof}

\begin{proof}[Proof of Theorem \ref{mthm}]
Write $f_i=\partial f/\partial x_i$ for any $i=1,\cdots,n$. If $f_i\notin \m$ for some $i$, as $f_i$ is weighted homogeneous, $f_i$ is a nonzero constant, so $S/J_f=0$ and there is nothing further to prove. Thus we may assume that $f_i\in \m$ for any $i=1,\cdots,n$. By Lemma \ref{lem:splitting}, $S/J_f$ is graded-isomorphic to $\K[z_1,\ldots,z_m]/J_h$. If $m=0$, then $S/J_f\cong\K$ and every derivation is zero. Otherwise, replacing $f$ by $h$ and relabeling the variables, we may assume that $\deg f\geq 2w_i$ for every $i=1,\cdots,n$.

Suppose, for a contradiction, that $S/J_f$ has a nonzero  derivation $\overline{D}$ of weight $-q<0$. By Lemma \ref{lem:lift}, $\overline{D}$ can be lifted to a nonzero derivation $D$ on $S$ of weight $-q$ such that $D(J_f)\subseteq J_f$. Write
\begin{equation*}
D=p_1\partial/\partial {x_1}+\cdots+p_n \partial/\partial {x_n},
\end{equation*}
where $p_i$ is either the zero polynomial or a weighted homogeneous polynomial with $\deg p_i-w_i=-q$ for $i=1,\cdots,n$. Set
$$g=D(f)=\sum_{i=1}^n p_i f_i.$$

We claim that $J_g\subseteq J_f$. Indeed, the claim follows from
\begin{align*}
\frac{\partial g}{\partial x_k}=& \sum_{i=1}^n \frac{\partial p_i  }{\partial x_k} f_i + \sum_{i=1}^n p_i\frac{\partial f_i  }{\partial x_k}
=\sum_{i=1}^n \frac{\partial p_i  }{\partial x_k} f_i + \sum_{i=1}^n p_i\frac{\partial f_k  }{\partial x_i}\\
=&\sum_{i=1}^n \frac{\partial p_i  }{\partial x_k} f_i + D(f_k) \in J_f
\end{align*}
for any $k=1,2,\cdots,n$.

If $g\neq 0$, then $g$ is a nonzero weighted homogeneous polynomial with $\deg g=\deg f-q<\deg f$. Note that each $f_i$ is weighted homogeneous of degree $\deg f-w_i\geq w_i>0$. Since $g\in J_f$ is nonzero, it follows that $\deg g>0$, contradicting Proposition \ref{prop:degree}.

Thus $g=\sum_i p_if_i=0$. By Lemma \ref{lem:koszul}, every $p_i$ belongs to $J_f$. It follows that $\overline{D}$ is a zero derivation on $S/J_f$, a contradiction.
\end{proof}

\end{document}